\documentclass{amsart}
\usepackage{amssymb,amsmath,amscd}
\usepackage{latexsym,bm,bbm,mathrsfs}
\usepackage{hyperref,graphicx}
\usepackage{mathtools}
\usepackage{stmaryrd}
\usepackage[all,pdf]{xy}
\newtheorem{thm}{Theorem}[section]
\newtheorem{quest}[thm]{Question}

\newtheorem{cor}[thm]{Corollary}
\newtheorem{prop}[thm]{Proposition}

\newtheorem{lem}[thm]{Lemma}

\numberwithin{equation}{section}

\renewcommand{\P}{\mathbb{P}}

\newcommand{\C}{{\mathbb{C}}}
\newcommand{\F}{\mathbb{F}}
\newcommand{\N}{\mathbb{N}}
\newcommand{\Q}{{\mathbb{Q}}}
\newcommand{\R}{{\mathbb{R}}}
\newcommand{\Z}{{\mathbb{Z}}}
\newcommand{\Lap}{\mathcal{L}}

\newcommand{\Gal}{\operatorname{Gal}}
\newcommand{\Aut}{\operatorname{Aut}}
\newcommand{\supp}{\operatorname{supp}}
\newcommand{\Ind}{\operatorname{Ind}}

\newcommand{\GL}{\operatorname{GL}}
\newcommand{\SL}{\operatorname{SL}}
\newcommand{\PSL}{\operatorname{PSL}}
\newcommand{\PGL}{\operatorname{PGL}}

\newcommand{\Tr}{\operatorname{Tr}}

\newcommand{\X}{\mathsf{X}}
\newcommand{\sS}{\mathsf{S}}

\begin{document}
\title{Can one hear the shape of a random walk?}
\author{Michael Larsen}
\email{mjlarsen@iu.edu}
\address{Department of Mathematics, Indiana University, Bloomington, IN, 47405, U.S.A.}
\thanks{The author was partially supported by NSF grants DMS-2001349 and DMS-2401098 and the Simons Foundation.}
\begin{abstract}
To what extent is the underlying distribution of a finitely supported unbiased random walk on $\Z$ determined by the sequence of times
at which the walk returns to the origin?
The main result of this paper is  that, in various senses, most unbiased random walks on $\Z$ are
determined up to equivalence by the sequence $I_1,I_2,I_3,\ldots$, where $I_n$ denotes the probability of being at the origin after $n$ steps.
We also give an application to an inverse problem from asymptotic representation theory.
The proof uses Laplace's method and a delicate Galois-theoretic analysis which ultimately depends on the classification of finite simple groups.
\end{abstract}
\maketitle
\section{Introduction}
In this paper, a \emph{random walk} will mean the sequence of partial sums of a series $\X_1+\X_2+\X_3+\cdots$
of independent and identically distributed finitely supported $\Z$-valued random variables.  It is \emph{unbiased} if its mean $E[\X_i]$ is $0$.
Its \emph{shape} will mean the common distribution of the $\X_i$, which is encoded in the generating function
\begin{equation}
\label{generating function}
\chi(t) = \sum_{k=-e}^f \kappa_k t^k,\ e,f\in \Z,
\end{equation}
where $\kappa_k = \Pr[\X_i=k]\ge 0$ for all $k$.
The shape of any random walk satisfies $\chi(1) = 1$, and the walk is unbiased if and only if $\chi'(1) = 0$.
The \emph{support}, denoted $\supp(\chi)$,
is the set of $s\in\Z$ such that $\kappa_s > 0$.  The \emph{degree}, denoted $\deg(\chi)$, is the degree of $\chi(t)$, 
regarded as a map of Riemann surfaces $\C^\times\to \C$; 
assuming $e$ and $f$ are positive, this is $n:=e+f$.
Two walks are \emph{equivalent} if they are the same up to an automorphism of $\Z$, i.e., if their shapes are the same up to possibly replacing $t$ by $t^{-1}$.
A random walk is
\emph{primitive} if $\chi$ cannot be expressed as the composition of two rational functions of degree $>1$.
The \emph{$n$th return probability} $I_n$ is the probability $\Pr[\X_1+\cdots+\X_n=0]$.
The \emph{spectrum} of the random walk will mean
the sequence of return probabilities.
Unbiased random walks are recurrent, so by the strong law of large numbers, $I_\bullet$ is determined by the set-valued random variable $\sS := \{n\in \N\mid \X_1+\cdots+\X_n=0\}$ by the formula
$$I_n = \lim_{N\to \infty} \frac{|[1,N]\cap \sS\cap (n+\sS)|}{|[1,N]\cap\sS|}$$
almost surely.

This paper is motivated by the following question:
\begin{quest}
\label{Drum}
Is a primitive unbiased random walk determined up to equivalence
by its spectrum?
\end{quest}

It is certainly not the case that random walks in general are determined up to equivalence by their spectra.
For instance, a walk whose support consists only of positive integers has $I_n=0$ for all $n>0$.  
Nor is it enough to assume that $e$ and $f$ are both positive in \eqref{generating function}.
Indeed, if $\chi(\lambda)=1$ (so that $\chi(\lambda t)$ is the shape of a different random walk)
the shapes $\chi(t)$ and $\chi(\lambda t)$ have the same spectrum.
Thus, for instance, $\frac47t^2 + \frac37t^{-1}$ and $\frac17t^2 + \frac67t^{-1}$
give rise to isospectral random walks.  This class of exceptions is eliminated by assuming the walk is unbiased.  
Another class of potential exceptions comes from the fact that, for all $n\neq 0$, random walks of shape $\chi(t)$ and $\chi(t^n)$ have the same spectrum.  The primitivity hypothesis excludes such examples for $n\neq \pm 1$. (Primitivity may be more than is needed; a stronger version of Question~\ref{Drum} would allow imprimitive walks but redefine equivalence to allow scaling by factors other than $\pm 1$.)

One way of thinking about  return probabilities is as  moments of the pushforward $\mu$ of the uniform distribution on $\R/2\pi \Z$ by $\theta\mapsto \chi(e^{i \theta})$.
For symmetric random walks (which are generally not primitive, since the symmetry implies that $\chi(t)$ can be written $P(t+t^{-1})$ for some polynomial $P$) Max Zhou \cite{Zhou}
used this point of view to prove that the spectrum determines the walk uniquely.
Indeed, $\mu$ is supported on the interval $[-1,1]\subset \R$ and is therefore determined by its moments.  By symmetry, it can be realized as the pushforward of the uniform distribution
on $[0,\pi]$ by $\chi(e^{i \theta})$, and since this map gives a bijection between a small interval $[0,\epsilon]$ in $\theta$ and a small interval $[\chi(e^{i\epsilon}),1]$,
$\mu$ actually determines $\chi$ on $[0,\epsilon]$.
The difficulty for non-symmetric walks is that $\mu$ is not supported in $\R$.  To follow the approach of \cite{Zhou}, it would be necessary to know
the expectations of $\chi^m \bar\chi^n$ or, equivalently, the probabilities
$$\Pr[\X_1+\cdots+\X_m = \X_{m+1}+\cdots+\X_{m+n}],$$
for all pairs of non-negative integers $m$ and $n$.  

The main theorem of this paper answers Question~\ref{Drum} except in certain special degrees.
\begin{thm}
\label{Main}
Suppose two primitive unbiased random walks have the same spectrum and the first has degree $n$.  If $n\neq 10$ and $n$ is 
not a perfect square, then the two walks are equivalent.
\end{thm}

The logic of the proof goes as follows:
$$\xymatrix{I_\bullet\ar[r]^-{\txt{Laplace's\\ method}}&L(s)\text{ at } \infty\ar[rr]^-{\txt{``Wick rotation''}}&&\tilde L(s)\text{ at }\infty
\ar[dd]_{\txt{inverse\\ Laplace \\ transform}}\\ \\
\{\gamma_1,-\gamma_2\}\text{ mod }K\ar[dd]^{\txt{Lagrange\\ inversion\\ theorem}}&\gamma_1-\gamma_2\text{ at }\infty\ar[l]_-{\txt{M\"uller's\\ Thm.}}&&\gamma_1(u)-\gamma_2(u)\text{ at }0\ar[ll]_{\txt{analytic\\ continuation}}\\ \\
\{[\chi^+],[\chi^-]\}\ar[rr]^-{\txt{monotonicity}}&&\{\chi(t),\chi(t^{-1})\}}$$
To flesh this out, in \S2 we consider the power series in $s$ given by the integral
\begin{equation}
\label{Laplace}
L(s) :=  \int_{-\pi}^\pi e^{s \chi(e^{i\theta})}d\theta.
\end{equation}
The spectrum $I_\bullet$ determines all the coefficients of this series and therefore the asymptotic expansion of the resulting function at $s=\infty$.
By ``Wick rotation'', we mean replacing \eqref{Laplace} by the integral
\begin{equation}
\label{Wick}
\tilde L(s) := \int_{-\infty}^\infty e^{-s\chi(e^x)}dx,
\end{equation}
whose expansion at $s=\infty$ can be read off from the expansion of \eqref{Laplace} at $s=\infty$.
A significant advantage of $\chi(e^x)$ over $\chi(e^{i\theta})$ is that it takes real values on $\R$.

Now $u = \chi(t)-1$ is convex for $t\in (0,\infty)$ and achieves its minimum value $0$ at $t=1$, so there are well-defined inverse functions $\alpha_1\colon [0,\infty)\to (0,1]$ and $\alpha_2\colon [0,\infty)\to [1,\infty)$, which are algebraic functions 
conjugate to one another over the field $K:=\C(u)$. More precisely, embedding $K$ in $K_\infty := K((u^{-1}))$, every element of the finite extension $K((u^{-1/k}))$ of $K_\infty$
defines a smooth function on $[a,\infty)$ for some $a$, where $u^{-1/k}$ is taken to be the positive real $k$th root of $u^{-1}$. 
In the case of the $\alpha_i$, $a$ can be taken to be $0$.

The integral \eqref{Wick} can be expressed as the Laplace transform of the  function $\gamma_1(u) - \gamma_2(u)$,
where $\gamma_1(u)$ and $\gamma_2(u)$ are the logarithmic derivatives of $\alpha_1(u)$ and $-\alpha_2(u)$, 
so they too are algebraic functions over $K$ in the same sense.
The Puiseux series expansion in $u$ which gives $\gamma_1-\gamma_2$ on $[0,\infty)$
is determined by the asymptotic behavior at $s=\infty$ of its Laplace transform. 

The Puiseux series in $u$ determines $\gamma_1-\gamma_2$ in a neighborhood $[0,\epsilon)$ of $0$ and therefore as an
analytic function along the positive real line, and this gives the Puiseux series at $\infty$, i.e., in $u^{-1}$.
We would like to determine $\gamma_1$ and $\gamma_2$ individually,
as far as this is possible.
At best, we can only hope to determine the pair $(\gamma_1,\gamma_2)$ up to translation by $\{(x,x)\mid x\in K\}$ and up to the exchange $(\gamma_1,\gamma_2)\leftrightarrow(-\gamma_2,-\gamma_1)$.
Writing $\chi(t) = \chi^-(t^{-1}) + \kappa_0 + \chi^+(t)$, where $\chi^{\pm}$ are polynomials with constant term $0$, we will show 
that knowing $\{\gamma_1,-\gamma_2\}$
up to translation in $K$  is enough to determine 
the unordered pair of  classes $\{[\chi^-],[\chi^+]\}$, where $[\phi]$ denotes the equivalence class of polynomials $\{\phi(\lambda t)\mid \lambda>0\}$.
This data, together with $\kappa_0 = I_1$, determines $\chi(t^{\pm})$ uniquely.

In \S3, we abstract the situation as follows. Suppose $K$ is a field of characteristic $0$ and $\gamma_1$ and $\gamma_2$ are conjugate elements
in the splitting field $M$ of the minimal polynomial of $\gamma_1$ over $K$.
In Galois theory, it is often but not always the case that $\gamma_1-\gamma_2$ determines $\{\gamma_1,-\gamma_2\}$ up to the
additive action of $K$.
In our case, we know two additional things about the set $\{\gamma_k\}$ of conjugates of $\gamma_1$. 
The permutation representation of $\Gal(M/K)$ acting on this set is primitive, and it contains an element, namely ``local monodromy at $\infty$,'' which has exactly two orbits, of sizes $e$ and $f$.

Peter M\"uller has classified \cite{Muller}  the possible permutation representations arising from primitive Laurent series by an argument which rests ultimately on the classification of finite simple groups. Unfortunately, there are 
cases within his classification scheme in which $\gamma_1-\gamma_2$ does \emph{not}
determine $\{\gamma_1,-\gamma_2\}$ modulo $K$. If $\beta_1,\ldots,\beta_5$ are the five roots of a quintic polynomial over $K$
with Galois group $S_5$, then 
$$\{\beta_i+\beta_j\mid 1\le i < j\le 5\}$$
is a set of $10$ conjugate elements $\{\gamma_k\}$ which generate $M = K(\beta_1,\ldots,\beta_5)$.
There are three different pairs $(i,j)$ so that $\gamma_i-\gamma_j = \beta_1-\beta_2$.

A second class of examples works similarly: let $K'$ denote a quadratic extension of $K$ and $P(x)\in K'(x)$ denote a polynomial with roots $\alpha_1,\ldots,\alpha_m$. Let $\beta_1,\ldots,\beta_m$ denote the
roots of $P^\sigma(x)$, where $\sigma$ generates $\Gal(K'/K)$. Let $\gamma_1,\ldots,\gamma_{m^2}$ denote all expressions of the form $\alpha_i+\beta_j$.
Some differences $\gamma_k-\gamma_l$ are of the form $\beta_2-\beta_1$, and again such differences do not determine the pair $(k,l)$.
Because of these examples, we cannot rule out the possibility of a negative answer to Question~\ref{Drum} based on primitive random walks of degree $n$ if $n=10$ or $n$ a perfect square.

Much of the complexity of this argument can be avoided by assuming that $e$ and $f$ are relatively prime, and here our result is optimal (and, moreover, can be proven without appealing to
classification.)

\begin{thm}
\label{clean}
Any  random walk with the same spectrum as an unbiased random walk with $(e,f)=1$ is equivalent to it up to an endomorphism of $(\Z,+)$.
\end{thm}

For general $e$ and $f$, one can still ask whether \emph{generically}, the spectrum determines $\chi$.
Regarding the coefficients $\kappa_i$ of $\chi$ as variables, the return probabilities $I_1,I_2,\ldots$ can be regarded as polynomial functions in the $\kappa_i$.  
For each positive integer $N$, the $N$-tuple $(I_1,I_2,\ldots,I_N)$ defines a morphism of algebraic varieties
to affine $N$-space.  This map has been considered before in connection with a conjecture of Olivier Mathieu \cite{Mat}, settled by Hans Duistermaat and Wilberd van der Kallen \cite{Dv}, asserting that for
each $e$ and $f$ there exists $N$ such that no non-zero power series maps to $(0,\ldots,0)\in \C^N$.
More recently, Daniel Erman, Gregory Smith, and Anthony V\'arilly-Alvarado \cite{ESV},
considering  $(I_1,I_2,\ldots,I_{e+f-1})$ as a map on the affine subspace 
given by $\kappa_{-e}=\kappa_f=1$,
proved that it is generically finite-to-one, and computed its degree, which turns out to be an Eulerian number.
In this paper, it is more natural to consider the affine space $X_{e,f}$ defined by the constraints
\begin{equation}
\label{affine}
\sum_{k=-e}^f \kappa_k = 1,\ \sum_{k=-e}^f k \kappa_k = 0.
\end{equation}

A strong version of Question~\ref{Drum} would ask if $(I_1,\ldots,I_N)$ is injective on $X_{e,f}$ for $N$ sufficiently large in terms of $e$ and $f$.
Here is a weaker version.

\begin{thm}
\label{generic}
Given $e$ and $f$, there exists a dense Zariski open subset $U$ of $X_{e,f}$ such that any random walk whose shape \eqref{generating function} lies in $U$
is determined up to equivalence by its spectrum.
\end{thm}

Together, Theorems \ref{Main}, \ref{clean}, and \ref{generic}
could be regarded as modest evidence in favor of a positive answer to Question~\ref{Drum}.
On the other hand, I recently learned from Swarnava Mukhopadhyay of his preprint  \cite{BGM} with Pieter Belmans and Sergey Galkin constructing pairs of unbiased random walks on $\Z^r$
with identical $I_\bullet$ which are not in any sense isometric. We seem to have independently hit upon the same title for our papers.

In the special case that $\chi$ has rational coefficients, we can express it as $\frac1d \sum_k w_k t^k$, where the $w_k$ are non-negative integers. If $V$ denotes the complex $d$-dimensional representation 
of $U(1)$ in which each weight $k$ appears with multiplicity $w_k$, we can identify $I_n$ with $d^{-n} \dim (V^{\otimes n})^{U(1)}$.
There has long been interest in the asymptotic behavior of the multiplicities of irreducible (or more generally, indecomposable) summands of $V^{\otimes n}$, where $V$ is an object in a tensor category.
See, for instance, \cite{Biane, CEO, PR, TZ} and the references therein. In particular, the asymptotics of the multiplicity of the trivial representation are well understood for determinant-$1$ faithful representations $V$ of compact Lie groups $G$, 
and it is known how to read $\dim V$, $\dim G$, and $\mathrm{rank}\,G$ from them. 

These multiplicities can also be understood as moments of the Sato-Tate measure $\mu_{G,V}$, i.e., the pushforward of the Haar measure of G by the character of the representation $V$. 
In the spirit of Question~\ref{Drum}, one might ask whether the sequence of moments of $\mu_{G,V}$ determines a compact subgroup $G\subset \SL(V)$ up to conjugation.
It is known \cite{AYY,LP} that the answer is no; even full knowledge of $\mu_{G,V}$ does not in general allow one to determine $G$ up to conjugation.
However, known counterexamples do not exclude the possibility that when $G\cong U(1)$, $\mu_{U(1),V}$ does determine $U(1)$ up to conjugation in $\SL(V)$.

The failure of $\mu_{G,V}$ to determine $(G,V)$ gives one of the earliest known mechanisms for constructing pairs of isospectral Riemannian manifolds, 
going back to the work of Toshikazu Sunada \cite{Sunada} in the case of finite groups $G$ 
(and, in fact, further back to a 1926 paper of Fritz Gassmann \cite{Gassmann}, which gives, in the setting of Dedekind zeta-functions, the earliest hints of isospectrality).  See also \cite{AYY,Sutton} for the case of connected Lie groups $G$.
This seems to give the most direct connection between Question~\ref{Drum} and the problem popularized in the classic paper of Mark Kac \cite{Kac}.  
However, there are other analogous features, for instance the fact that in each case, information on the ``geometric'' side is determined by the asymptotic behavior of the ``spectrum''. 

Motivated by Pierre Deligne's Sato-Tate theorem \cite[Theorem~3.5.3]{Weil} for the distribution of values of families of complete exponential sums over finite fields,
Nicholas Katz many years ago suggested
the problem of how much information about the pair $(G,V)$ is encoded in $\mu_{G,V}$ and particularly in its moments.
Richard Pink explained to me how to use Laplace's method to compute the asymptotics of the dimensions of the moments of $\mu_{G,V}$.
I am grateful to both of them.

\section{Laplace's method}
In this section $\chi$ is still defined by \eqref{generating function}, with $\kappa_k\ge 0$ for all $k$,
$\chi(1) = 1$, and $\chi'(1) = 0$, but
we relax the hypothesis that $\chi$ is primitive, making only the weaker assumption that 
the greatest common divisor of the elements of 
$$S:=\supp(\chi)=\{k\mid \kappa_k>0\}$$ 
is $1$.

We write
$$I_n = (2\pi)^{-1}\int_0^{2\pi} \chi(e^{ix})^n dx = (2\pi)^{-1}\int_{-\pi}^{\pi} \chi(e^{ix})^n dx$$
for $n\in \N$.
By the triangle inequality, $|\chi(e^{ix})|\le 1$ for all $x\in\R$,
so $|I_n|\le 1$ for all $n$.
Therefore, we can exchange integral and sum to show
$$(2\pi)^{-1}e^{-s}\int_{-\pi}^\pi e^{s \chi(e^{ix})} dx = (2\pi)^{-1}e^{-s}\int_{-\pi}^\pi \sum_{n=0}^\infty \frac{s^n}{n!}\chi(e^{ix})^n dx= e^{-s}\sum_{n=0}^\infty \frac{s^n I_n}{n!}.$$
We call this expression $L(s)$.
For all integers $n\ge 2$, let
$$J_n = \sum_{k = -e}^f \kappa_k k^n.$$

I learned how to prove the following proposition from Richard Pink in the 1980s.

\begin{prop}
\label{L asymptotics}
There exist real numbers $1=A_0,A_1,A_2,\ldots$ which can be expressed polynomially in $\{J_2^{-n/2}J_n\mid n\ge 3\}$ such that
for all $m\in\N$,
$$L(s) = \sqrt{J_2/2\pi}\sum_{l=0}^m A_l s^{-l-1/2} + o(s^{-m-1/2})$$
as $s\to\infty$.  
\end{prop}
\begin{proof}
As
$$s(\chi(e^{ix})-1) = s\Bigl(\frac{-J_2 }2x^2 + \sum_{n=3}^\infty \frac{i^n J_n x^n}{n!}\Bigr),$$
substituting $u = \sqrt{J_2 s} x$, we have
\begin{equation}
\label{L_f}
L(s) = (2\pi)^{-1}(J_2 s)^{-1/2} \int_{-\sqrt{J_2 s} \pi}^{\sqrt{J_2 s}\pi} e^{-\frac{u^2}2}
\exp\Bigl(s \chi\Bigl(e^{iu/\sqrt{J_2 s}}\Bigr) - s + \frac{u^2}2\Bigr) du.
\end{equation}
Fix $\epsilon < 1/(12m+6)$.  By the remainder form of Taylor's theorem, for $|u| < s^\epsilon$, we have
$$e^{iu/\sqrt{J_2 s}}-\sum_{n=0}^{2m+2}\frac 1{n!}\Bigl(\frac{iu}{\sqrt{J_2 s}}\Bigr)^n = O(s^{-m-3/2+(2m+3)\epsilon}) = o(s^{-m-1}),$$
so
$$s \chi\bigl(e^{iu/\sqrt{J_2 s}}\bigr) - s + \frac{u^2}2 = \sum_{n=3}^{2m+2}\frac {i^nJ_n}{n!J_2^{n/2}} u^n s^{1-n/2}+ o(s^{-m}).$$
Since this sum is $O(u^3 s^{-1/2}) = o(s^{-m/(2m+1)})$,
by a second application of Taylor's theorem,
\begin{equation}
\label{Taylor}
\begin{split}
\exp\Bigl(s \chi\Bigl(e^{{iu}/{\sqrt{J_2 s}}}\Bigr) - s + \frac{u^2}2\Bigr) &= \sum_{r=0}^{2m}\frac 1{r!}\Bigl(s \chi\Bigl(e^{iu/\sqrt{J_2 s}}\Bigr) - s + \frac{u^2}2\Bigr)^r\ + o(s^{-m})\\
&=\sum_{r=0}^{2m}\frac 1{r!}\Bigl(\sum_{n=3}^{2m+2}\frac {J_n}{n!J_2^{n/2}} (iu)^n s^{1-n/2}\Bigr)^r+o(s^{-m}).
\end{split}
\end{equation}
Expanding the inner sum on the right hand side, we have a polynomial in $w = iu$ and $z = s^{-1/2}$ where the $w$-degree of any monomial is at most three times the $z$-degree.
Therefore, for suitable $g_{j,k}$
\begin{equation}
\label{truncated}
\int_{-s^\epsilon}^{s^\epsilon} e^{-u^2/2} \exp\Bigl(s\chi\Bigl(e^{{iu}/{\sqrt{J_2 s}}}\Bigr) - s + \frac{u^2}2\Bigr) du = \sum_{j=0}^{6m} \sum_{k=0}^{2m} g_{j,k} i^js^{-k/2}\int_{-s^\epsilon}^{s^\epsilon} e^{-u^2/2} u^j du + o(s^{-m}).
\end{equation}
We have
$$\int_{-s^\epsilon}^{s^\epsilon} e^{-u^2/2} u^j du
=\begin{cases}
\sqrt{2\pi}(j-1)!!+o(s^{-m})&\text{ if $j$ is even,}\\
0&\text{ if $j$ is odd.}
\end{cases}
$$
Since only even values of $j$ contribute and $g_{j,k} = 0$ when $j+k$ is odd, we can limit the sum
\eqref{truncated} to terms for which $j$ and $k$ are both even; the sign of the contribution is $(-1)^{j/2}$.
The real part of $\chi(e^{ix_0})$ is strictly less than $1$ unless $e^{i k x_0} = 1$ for all $k\in S$.
As the greatest common divisor of the elements of $S$ is $1$, this occurs only if $x_0\in 2\pi \Z$.  Therefore, on $[-\pi,\pi]$, the real part of $\chi(e^{ix})$ achieves its maximum only at $x=0$, and we can write
\begin{equation}
\label{real bound}
\Re(\chi(e^{ix})) \le 1-\delta x^2
\end{equation}
for some $\delta > 0$.  It follows that the integrand in \eqref{L_f} is at most $e^{-\delta s^\epsilon}$ outside $[-s^\epsilon,s^\epsilon]$, so this part of the integral contributes $o(s^{-m})$.
\end{proof}
Next, we replace $\chi(e^{ix})$ by $\chi(e^x)$ and $L(s)$ by
$$\tilde L(s) = (2\pi)^{-1} \int_{-\infty}^\infty e^{-s(\chi(e^x)-1)}dx.$$
As $\chi(e^x)$, regarded as a function $\R\to \R$,  is convex, with its (unique) minimum at $x=0$, 
this integral converges absolutely for all $s$ in the right half plane.
In analogy with Proposition~\ref{L asymptotics}, we have:
\begin{prop}
\label{tilde asymptotics}
With $1=A_0,A_1,\ldots$ defined as before and $m\in\N$,
$$\tilde L(s) = \sqrt{J_2/2\pi} \sum_{l=0}^m (-1)^l A_l s^{-l-1/2} + o(s^{-m-1/2}).$$
\end{prop}
The proof is essentially the same as before except that we need $\chi(e^x) \ge 1+\delta x^2$ instead of \eqref{real bound}.
The former statement follows from  the weighted arithmetic-geometric mean inequality:
$$ \frac {d^2\chi(e^x)}{dx^2}= \sum_k a_k k^2 e^{kx} \ge \prod_k (k^2 e^{kx})^{a_k} = \prod_k k^{2a_k}.$$
The analogue of \eqref{Taylor} is
\begin{equation*}
\begin{split}
\exp\bigl(-s \chi\Bigl(e^{u/{\sqrt{J_2 s}}}\Bigr) + s + \frac{u^2}2\bigr) &= \sum_{r=0}^{2m}\frac 1{r!}\bigl(-s \chi\Bigl(e^{u/{\sqrt{J_2 s}}}\Bigr) + s + \frac{u^2}2\bigr)^r\ + o(s^{-m})\\
&=\sum_{r=0}^{2m}\frac 1{r!}\Bigl(-\sum_{n=3}^{2m+2}\frac {J_n}{n!J_2^{n/2}} u^n s^{1-n/2}\Bigr)^r+o(s^{-m}).
\end{split}
\end{equation*}
Expanding,
each monomial contributing $u^j s^{-k/2}$ comes from the term 
$$\Bigl(-\sum_{n=3}^{2m+2}\frac {J_n}{n!J_2^{n/2}} u^n s^{1-n/2}\Bigr)^r,$$
with $r = (j-k)/2$, giving a sign of 
$(-1)^r = (-1)^{\frac{j-k}2}$. In \eqref{Taylor}, the corresponding sign from
$$\Bigl(\sum_{n=3}^{2m+2}\frac {J_n}{n!J_2^{n/2}} (iu)^n s^{1-n/2}\Bigr)^r$$
was $(-1)^{j/2} = i^j$. Thus, we pick up an additional factor of $(-1)^{(j-k)/2}/(-1)^{j/2} = (-1)^{-k/2}$. Since only even $k$ contribute and 
we are looking at the $s^{-l-1/2}$  term, where $k=2l$, this gives an overall sign change of $(-1)^l$ relative to Proposition~\ref{L asymptotics},
as claimed.
Combining Propositions~\ref{L asymptotics} and \ref{tilde asymptotics}, the spectrum determines, first the $A_i$, and then
the asymptotic series for
$\tilde L(s)$ as $s\to \infty$.  Define $\phi_1,\phi_2\colon [0,\infty)\to [0,\infty)$ to be the bijective functions $\phi_1(x) = \chi(e^x)-1$ and 
$\phi_2(x) = \chi(e^{-x}) - 1$.  Writing
$$\int_{-\infty}^\infty e^{-s (\chi(e^x)-1)} dx = \int_0^\infty e^{-s\phi_1(x)} dx + \int_0^{\infty} e^{-s \phi_2(x)} dx.$$
and substituting $u= \phi_1(x)$ in the first summand on the right hand side and $u = \phi_2(x)$ in the second, we obtain
$$2\pi \tilde L(s) = \int_{-\infty}^\infty e^{-s (\chi(e^x)-1)} dx = \int_0^\infty e^{-su} \frac{du}{\phi_1'(\phi_1^{-1}(u))} + \int_0^\infty e^{-su} \frac{du}{\phi_2'(\phi_2^{-1}(u))}.$$
We can write this as the Laplace transform $\Lap(\gamma_1(u)-\gamma_2(u))$, where
$$\gamma_1(u)= \phi_1'(\phi_1^{-1}(u))^{-1} = \frac 1{\alpha_1(u)\chi'(\alpha_1(u))}$$
and
$$\gamma_2(u)= -\phi_2'(\phi_2^{-1}(u))^{-1} = \frac{1}{\alpha_2(u)\chi'(\alpha_2(u))}$$
The asymptotic growth of $\Lap(\gamma_1-\gamma_2)$ at $s=\infty$ determines the Puiseux expansion of $\gamma_1(u)-\gamma_2(u)$ at $u=0$, which is to say, at $z=1$.
Since an algebraic function is determined by its Puiseux expansion,
Proposition~\ref{tilde asymptotics} now implies that the sequence $A_0,A_1,\ldots$ determines $\gamma_1-\gamma_2$.
\section{Primitive Galois groups}
Let $K$ be any field of  characteristic $0$ and $M/K$ a finite Galois extension with group $G$.  Let $X=\{\delta_1,\ldots,\delta_n\}$ and $Y= \{\varepsilon_1,\ldots,\varepsilon_m\}$ denote $G$-orbits in $M$ 
each of which generates $M$ over $K$. We assume that $K(\delta_i)$ and $K(\varepsilon_j)$ are extensions of $K$ with no intermediate fields, so $G$ acts primitively on $X$ and $Y$.

\begin{prop}
\label{two 2-trans}
Suppose $G$ acts $2$-transitively on $X$ and $Y$. Let $G_i$ (resp. $H_i$) denote the stabilizer of $\delta_i$ (resp. $\varepsilon_i$).
Suppose that if  
\begin{equation}
\label{2-point}
G_1\cap G_2 = H_1\cap H_2,
\end{equation}
then $\{H_1,\ldots,H_m\} = \{G_1,\ldots,G_n\}$.
If $\varepsilon_1-\varepsilon_2 = \delta_1-\delta_2$,
then either $\varepsilon_1-\delta_1\in K$ or $\varepsilon_1+\delta_2\in K$.
\end{prop}

\begin{proof}
Replacing each $\delta_i$ by $\delta_i - \frac 1n\sum_i \delta_i$ and each $\varepsilon_i$ by $\varepsilon_i - \frac 1m\sum_i \varepsilon_i$,
we may assume without loss of generality that  $\Tr_{M/K} \delta_i=\Tr_{M/K}\varepsilon_j=0$ for all $i$ and $j$.

Let $L := K(\delta_1)$, so $\Gal(M/L) = G_1$.
Let $V := \Q X$ denote the $\Q$-vector space with basis $X$. As a $G$-representation, it is $\Ind_{G_1}^G \Q$.
We denote the quotient of $V$ by its $1$-dimensional $G$-invariant subspace by $V_0$ and identify it with the set of $\Q$-valued functions on $X = G/G_1$ summing to $0$.
It is well known that the $2$-transitivity of the action on $X$ implies the absolute irreducibility of $V_0$.
Indeed, $V\cong V_0\oplus \Q$ as $G$-representation, but the inner product of the character of $V_0$ with itself is easily seen to be $2$,
so both factors must be absolutely irreducible. 

Consider the map $\phi\colon \ker\Tr_{L/K}\otimes V_0\to M$ of $K[G]$-modules  given by 
$$\gamma\otimes \sum_{i=1}^n a_i [g_i G_1] \mapsto \sum_{i=1}^n a_i g_i(\gamma),$$
where $g_i G_1$ is the set of elements of $G$ sending $\delta_1$ to $\delta_i$.
As $\ker\Tr_{L/K}\otimes V_0$ is isotypic, $\ker\phi$ must be of the form $L_0\otimes V_0$, where $L_0$ is a $K$-subspace of $\ker \Tr_{L/K}$.
However, $\gamma\otimes V_0\subset \ker \phi$ implies $\gamma\in M^G=K$. Since $\gamma$ has trace $0$, it must be $0$, implying $\phi$ is injective.

If  $\delta_1-\delta_2 = \delta_i-\delta_j$, then
$$\phi(\delta_1\otimes ([g_1G_1]-[g_2G_1] -[g_i G_1] + [g_j G_1])) = 0,$$
so $i=1$ and $j=2$. Thus, 
$$\Gal(M/K(\delta_1-\delta_2)) = G_1 \cap G_2.$$
Applying the same reasoning to the $\varepsilon_i$,
$$H_1\cap H_2 = \Gal(M/K(\varepsilon_1-\varepsilon_2)) = \Gal(M/K(\delta_1-\delta_2)) = G_1 \cap G_2,$$
so by hypothesis, $\{H_1,\ldots,H_m\} =  \{G_1,\ldots,G_n\}$. Writing $G_i$ for $H_1$ and $G_j$ for $H_2$, and defining $\zeta:= g_i^{-1}\varepsilon_1\in \ker \Tr_{L/K}$, we obtain
$$\phi(\delta_1 \otimes ([g_1 G_1] - [g_2 G_1])) = \phi(\zeta\otimes ([g_i G_1] - [g_j G_1])).$$
By the injectivity of $\phi$, this implies that for some $c\in \Q^\times$, $\delta_1 = c\zeta$ and 
$$c[g_1 G_1] - c[g_2 G_1] = [g_i G_1] - [g_j G_1].$$
Thus $c=\pm 1$.
If $c=1$, then $\delta_1 = \zeta$ and $g_i G_1 = g_1 G_1$, so $\varepsilon_1 = g_i\zeta = \zeta  = \delta_1$. If $c=-1$, then $\delta_1 = -\zeta$ and $g_2 G = g_i G$, so $\delta_2 = g_2 \delta_1 = - g_2 \zeta = - g_2 g_i^{-1} \varepsilon_1 = -\varepsilon_1$.

\end{proof}

The following two results deal with cases in which the same group has two non-isomorphic $2$-transitive actions and shows that \eqref{2-point} cannot hold in these cases.
There do exist groups which admit two non-isomorphic $2$-transitive actions for which \eqref{2-point} is satisfied, but, as we prove below, the local monodromy condition we impose on $\Gal(M/K)$ 
rules out such cases.

\begin{lem}
\label{2-primitive}
If $G$ acts $3$-transitively on $X$ and $2$-transitively on $Y$, then for all distinct $\delta_1,\delta_2\in X$ and distinct $\varepsilon_1,\varepsilon_2\in Y$, the equality \eqref{2-point}
implies that the actions on $X$ and $Y$ are isomorphic.
\end{lem}

\begin{proof}
We have already seen that we must have 
$$[G:G_1\cap G_2] = |X|(|X|-1) = |Y|(|Y|-1),$$ 
so the equality \eqref{2-point}  implies $|X| = |Y|$.
Now,
$$G_1\cap G_2 \subseteq G_1\cap H_1 \subseteq G_1.$$
Since $G_1$ acts $2$-transitively on $X\setminus\{\delta_1\}$, it acts primitively,  so $G_1\cap G_2$ is a maximal proper subgroup of $G_1$, implying that equality must hold in one of these two inclusions.
If equality holds in the first inclusion, then the orbit of $\delta_1$ under $H_1$ has cardinality
$$[H_1: G_1\cap H_1] = [G_1:G_1\cap G_2] = |X|-1,$$
so the complement of the $H_1$-orbit of $\delta_1$ is a single element $\delta_i\in X$, and $H_1$ must stabilize $\delta_i$. This implies $H_1 = G_i$.
Otherwise, $G_1\cap H_1 = G_1$, implying $G_1 = H_1$.
\end{proof}

\begin{prop}
\label{PSL}
If $p$ is prime, $q$ is a power of $p$, $k\ge 3$ with equality only if $q=4$,
$$\PSL_k(q) \le G \le \PGL_k(q)\rtimes \Gal(\F_q/\F_p),$$
and $X = \F_q\P^{k-1}$ with the standard $G$-action on points,
and $Y$ is a $2$-transitive action of $G$ not isomorphic to $X$, then $Y$ must be the standard action of $G$ on hyperplanes in $\P\F_q^{k-1}$, and \eqref{2-point} cannot hold.
\end{prop}

\begin{proof}
By \cite[p. 23]{Atlas}, $\PSL_3(4)$ has, up to conjugation, exactly $2$ subgroups of index $21$, the point stabilizer and the line stabilizer, so we may assume
$k\ge 4$.
By \cite{LZ}, every non-trivial character of $\PSL_k(q)$ has degree $\ge q^{k-1}-1$, so every subgroup of $\PSL_k(q)$ has index $\ge q^{k-1} > |X|/2$.
This implies $G_1^\circ := G_1\cap \PSL_k(q)$ has index $|X|$.
Thus, $G_1 = N_G G_1^\circ$, which means the action of $G$ is determined by the action of $\PSL_k(q)$.
It suffices, therefore, to prove that the only $2$-transitive actions of $\PSL_k(q) = \Aut(X)$ on $\frac{q^k-1}{q-1}$ points are the actions on points and on hyperplanes of $X$.
Henceforth, we assume $G = \Aut(X)$.

Next we observe that $|X|^{3/2} \le 4 q^{(3/2)(k-1)} < \frac{q^{k^2-1}}4 < |G|$, so $|G_1|^3 \ge |G|$. The possibilities for $(G=\PSL_k(q),G_1)$ are therefore tabulated in Proposition 4.7 and Table 7 of \cite{AB}. 
We claim that the only cases in which $G_1$ has index $|X|$ in $G$ belong to class $\mathcal{C}_1$ (stabilizers of subspaces), and the subspace must be of dimension $1$ or $k-1$.
We make use of Zsigmondy's theorem, which asserts that for $r>2$ and $(q,r)\neq (2,6)$, $q^r-1$ has a prime divisor $\ell_r$ which does not divide $q^s-1$ for any $s\in [2,r]$ (so, in particular, $\ell_r>r$.
There may be more than one, in which case we choose the largest one.

When $(q,k)=(2,7)$, we take $\ell := \ell_5 = 31$; otherwise,
we take  $\ell := \ell_{k-1}\ge k$. 
In every case $\ell$ divides $|G|$ but not $|X|$ and  therefore divides $|G_1|$.
In cases (ii)--(v) of \cite[Proposition 4.7]{AB}, $|H|$ is not divisible by $\ell$. Case (i) subdivides into stabilizers of non-degenerate forms and stabilizers of subspaces. For stabilizers of subspaces $H$, if $k$ is even,
 $\ell_{k-1}$ does not divide $|H|$ and if it is odd, $\ell_{k-2}$ does not divide $|H|$ (by the product formulas for the orders of orthogonal and symplectic groups, \cite[p. xvi]{Atlas}).
For stabilizers of subspaces where dimension and codimension are both at least $2$, $\ell$ does not divide $|H|$ except possibly when $\ell = \ell_{r-2}$, i.e., in the case $(q,k) = (2,7)$ where the subspace is of dimension $2$ or $5$. In this case, the codimension is $21\cdot 127$, while $|X| = 127$.  The Table 7 cases with $G = \PSL_k(q)$ are $(\PSL_5(7),U_4(2))$, $(\PSL_4(4),A_6)$, $(\PSL_6(3),M_{11})$, and $(\PSL_4(2),A_7)$.
In the first three cases, $|H|$ is not divisible by $\ell_3 = 19$, $\ell_3=7$, and $\ell_3 = 13$ respectively. For the last case, the index of $[G:H] = 8\neq 15 = |X|$.

Finally, we observe that the intersection of two hyperplane stabilizers stabilizes one point (namely the intersection of the hyperplanes, when they are lines in $\P^2$) and no points at all (in the case $k\ge 4$), so it cannot 
in any case equal the intersection of two distinct point stabilizers.

\end{proof}

\begin{prop}
\label{two-orbit}
Let $G$ be a finite group acting faithfully and primitively on a set $X$ such that the size $n$ of $X$ is not a perfect square and $n\neq 10$.
If $G$ has an element with exactly two orbits in $X$, of lengths $f\ge e>1$, then the action on $X$ is $2$-transitive. 
Moreover if $G$ has two such actions, they must be isomorphic.
\end{prop}
\begin{proof}
The determination of primitive permutation groups with an element whose orbit consists of exactly two cycles (of lengths $f\ge e$) is due to M\"uller \cite[Theorem~3.3]{Muller}
and reproduced here in Tables 1--3,  which present affine actions, product actions, and sporadic actions with some added information which is useful for our purposes:
the order of $G$, the number of simple factors in its Jordan-H\"older series, and the possible values of $1/e + 1/f$.
To be clear, the classification in these tables is up to group isomorphism, not inner automorphism.
Note also that there is a misprint in the table in \cite{Muller}, where sporadic case (e) is erroneously stated to have $n=19$.

Excluding cases where $n=10$ or $n$ is a square, the affine possibilities are cases (b) and (d) and the second subcase of (e).
Note that we have excluded case (a) from this list because $e=1$, case (c) because $n$ is square, and all  the other subcases of case (e), again because $n$ is square. 
The product cases in \cite{Muller} are all excluded by our hypothesis that $n$ is not a square.
Finally, our subset of the list of sporadic cases consists of cases (a), (d), and (f)--(i).

\begin{table}[htbp]
\centering
\caption{Affine Actions}
\label{tab:affine_actions}
\begin{tabular}{llccccr}
\hline
Label & Group & Order & \begin{tabular}[c]{@{}c@{}}Simple\\Factors\end{tabular} & $n$ & $e$ & $1/e+1/f$ \\
\hline
(a) & $> \F_{p^m}\rtimes\GL_{m/t}(p^t)$& $\approx p^{m+m^2/t}$ & $0$ or $1$ & $p^m$ & $1$ & $\frac{p^m}{p^m-1}$ \\[0.5ex]
(b) & $\F_2^2\rtimes \GL_2(2)$ & $24$ & $0$ & $4$ & $2$ & $1$ \\
& $\F_2^3\rtimes \GL_3(2)$ & $1344$ & $1$ & $8$ & $2$ & $2/3$ \\
& $\F_2^4\rtimes \GL_4(2)$ & $322560$ & $1$ & $16$ & $2$ & $4/7$ \\
& $\F_3^2\rtimes \GL_2(3)$ & $432$ & $0$ & $9$ & $3$ & $1/2$ \\
& $\F_5^2\rtimes \GL_2(5)$ & $12000$ & $1$ & $25$ & $5$ & $1/4$ \\
& $\F_p^m\rtimes \GL_m(p)$ &$\approx p^{m^2+m}$ & $1$ & $p^m$ & $p$ & $\frac{p^{m-2}}{p^{m-1}-1}$ \\[0.5ex]
(c) & $\F_p^2\rtimes N,\;p>2$ & $\le 2(p-1)p^2$ & $0$ & $p^2$ & $p$ & $\frac{1}{p-1}$ \\[0.5ex]
(d) & $\F_2^m\rtimes\GL_m(2)$ & $\approx 2^{m^2+m}$ & $1$ & $2^m$ & $4$ &$\frac{2^{m-4}}{2^{m-2}-1}$ \\[0.5ex]
(e) & $A_4$ & $12$ & $0$ & $4$ & $2$ & $1$ \\
& $\F_8\rtimes(\F_8^\times\rtimes C_3)$ & $96$ & $0$ & $8$ & $2$ & $2/3$ \\
& $\F_9\rtimes(\F_9^\times\rtimes C_2)$ & $108$ & $0$ & $9$ & $3$ & $1/2$ \\
& $\F_{16}\rtimes(C_5\rtimes C_4)$ & $320$ & $0$ & $16$ & $8$ & $1/4$ \\
& $\F_{16}\rtimes(\F_{16}^\times \rtimes C_4)$ & $512$ & $0$ & $16$ & $8$ & $1/4$ \\
& $\F_{16}\rtimes(C_3^2\rtimes C_4)$ & $576$ & $0$ & $16$ & $8$ & $1/4$ \\
& $\F_{16}\rtimes(\SL_2(4)\rtimes C_2)$ & $1920$ & $1$ & $16$ & $8$ & $1/4$ \\
& $\F_{16}\rtimes(\GL_2(4)\rtimes C_2)$ & $1152$ & $1$ & $16$ & $8$ & $1/4$ \\
& $\F_{16}\rtimes A_6$ & $5760$ & $1$ & $16$ & $8$ & $1/4$ \\
& $\F_{16}\rtimes \GL_4(2)$ & $322560$ & $1$ & $16$ & $8$ & $1/4$ \\
& $\F_{16}\rtimes(S_3^2\rtimes C_2)$ & $1152$ & $0$ & $16$ & $4$ or $8$ & $1/3$ or $1/4$ \\
& $\F_{16}\rtimes S_5$ & $1920$ & $1$ & $16$ & $4$ or $8$ & $1/3$ or $1/4$ \\
& $\F_{16}\rtimes S_6$ & $11520$ & $1$ & $16$ & $4$ or $8$ & $1/3$ or $1/4$ \\
& $\F_{16}\rtimes A_7$ & $40320$ & $1$ & $16$ & $2$ or $8$ & $4/7$ or $1/4$ \\
& $\F_{25}\rtimes G_1$ & $2400$ & $0$ & $25$ & $5$ & $1/4$ \\
\hline
\end{tabular}
\end{table}

\begin{table}[htbp]
\centering
\caption{Product Actions}
\label{tab:product_actions}
\begin{tabular}{llccccr}
\hline
Label & Group & Order & \begin{tabular}[c]{@{}c@{}}Simple\\Factors\end{tabular} & $n$ & $e$ & $1/e+1/f$ \\
\hline
(a) & $S_r^2\rtimes C_2$ ($r\ge 5$) & $2r!^2$ & $2$ & $r^2$ & $ar$ &$\frac 1{a(r-a)}$ \\[0.5ex]
& & & & & $(a,n)=1$ & \\[0.5ex]
(b) & $\PGL_2(p)^2\rtimes C_2$& $2(p^3-p)^2$  & $2$ & $(p+1)^2$ & $p+1$ & $1/p$ \\
& ($p\ge 5$) & & & & & \\
\hline
\end{tabular}
\end{table}

\begin{table}[htbp]
\centering
\caption{Sporadic Actions}
\label{tab:sporadic_actions}
\begin{tabular}{llccccr}
\hline
Label & Group & Order & \begin{tabular}[c]{@{}c@{}}Simple\\Factors\end{tabular} & $n$ & $e$ & $1/e+1/f$ \\
\hline
(a) & $A_5$ & $60$ & $1$ & $5$ & $1$ or $2$ & $5/4$ or $5/6$ \\
& $S_5$ & $120$ & $1$ & $5$ & $1$ or $2$ & $5/4$ or $5/6$ \\
& $A_n$ & $n!/2$ & $1$ & $n$ & $1,\ldots,\lfloor n/2\rfloor$ & $\frac n{e(n-e)}$\\
& ($n\ge 6$) & & & & & \\
& $S_n$ & $n!$ & $1$ & $n$ & $1,\ldots,\lfloor n/2\rfloor$ & $\frac n{e(n-e)}$\\
& ($n\ge 6$) & & & & & \\[0.5ex]
(b) & $A_5$ & $60$ & $1$ & $10$ & $5$ & $2/5$ \\
& $S_5$ & $120$ & $1$ & $10$ & $5$ & $2/5$ \\[0.5ex]
(c) & $>\PSL_2(p)$ & $\le p^3-p$ & $0$ or $1$ & $p+1$ & $1$ & $\frac{p+1}{p}$ \\[0.5ex]
(d) & $>\PSL_m(q)$ & $\approx q^{m^2-1}$ & $1$ & $\frac{q^m-1}{q-1}$ & $\frac{q^m-1}{2(q-1)}$ & $\frac{4(q-1)}{q^m-1}$ \\
& ($q$ odd) & & & & & \\[0.5ex]
(e) & $M_{10}$ & $720$ & $1$ & $10$ & $2$ & $5/8$ \\
& $M_{10}\rtimes C_2$ & $1440$ & $1$ & $10$ & $2$ & $5/8$ \\[0.5ex]
(f) & $\PSL_3(4)\rtimes C_2$ & $40320$ & $1$ & $21$ & $7$ & $3/14$ \\
& $\PGL_3(4)\rtimes C_2$ & $80640$ & $1$ & $21$ & $7$ & $3/14$ \\[0.5ex]
(g) & $M_{11}$ & $7920$ & $1$ & $12$ & $1$ or $4$ & $12/11$ or $8/3$ \\[0.5ex]
(h) & $M_{12}$ & $95040$ & $1$ & $12$ & $1$, $2$, $4$, or $6$ & $\frac{12}{11}$, $\frac{3}{5}$, $\frac{3}{8}$, or $\frac{1}{3}$ \\[0.5ex]
(i) & $M_{22}$ & $443520$ & $1$ & $22$ & $11$ & $2/11$ \\
& $M_{22}\rtimes C_2$ & $887040$ & $1$ & $22$ & $11$ & $2/11$ \\[0.5ex]
(j) & $M_{24}$ & $244823040$ & $1$ & $24$ & $1$, $3$, or $12$ & $\frac{24}{23}$, $\frac{8}{21}$, or $\frac{1}{6}$ \\
\hline
\end{tabular}
\end{table}

For the affine cases, the action of $G$ is $2$-transitive because $\GL_n(q)$ acts transitively on the non-zero vectors of $\F_q^n$.
For the sporadic list, $A_n$  (resp. $S_n$) is well-known to act $(n-2)$-transitively (resp. $n$-transitively) in its natural permutation action, so case (a) is always $3$-transitive.
Cases (g)--(j) are likewise $3$-transitive; see \cite[Table 7.4]{Cameron}.
For cases (d) and (f), the fact that $\PSL_k(F)$ acts $2$-transitively on $F\P^{k-1}$ for every field $F$ and every $k\ge2$ is immediate from the fact that if 
$\{e_1, e_2\}$ is a linearly independent set in $F^k$ , there exists an invertible linear transformation fixing $e_1$ and $e_2$ and therefore a determinant-$1$ transformation fixing $e_1$ and $Fe_2$. 

Next we prove that in each case there is no other $2$-transitive action with the same $2$-point stabilizer $G_1\cap G_2$. In the affine cases, every maximal subgroup $H$ of $G$ either contains 
the normal subgroup $N := \F_q^n$ or maps onto $G/N$. In the former case, $H$ and all its conjugates contain $N$, so the intersection of two conjugates cannot fix a point in the original affine action.
If they map onto $G/N$, then since this quotient acts irreducibly on $\F_q^n$, $H$ must be the image of a section of the quotient map $G\to G/N$. To prove that all such sections are conjugate, it suffices to
prove that $H^1(G/N,N) = 0$ in group cohomology. Let $Z\subset \GL_n(q)\lhd G/N$ denote the group of scalar matrices. In the inflation-restriction sequence
$$0 \to H^1((G/N)/Z, N^Z)  \to H^1(G/N,N) \to H^1(Z,N)^{G/N}$$
we have $H^1((G/N)/Z, N^Z) = H^1((G/N)/Z, (0)) = 0$ and $H^1(Z,N) = 0$ since the orders of $Z$ and $N$ are coprime. Therefore, all possible $H$ are conjugate to $G_1$.

For the sporadic cases, by Lemma~\ref{2-primitive}, it suffices to check that the action of $G$ on $X$ is $3$-transitive. 
As $A_n$ acts $(n-2)$-transitively and $S_n$ $n$-transitively, this holds for $n\ge 5$, giving case (a). Likewise, the actions in cases (g), (h), (i), and (j) 
are all $3$-transitive \cite[Table 7.4]{Cameron}.
Cases (d) and (f) follow from Proposition~\ref{PSL}.
\end{proof}

\begin{prop}
\label{Forced 2-transitivity}
Let $G_1$ and $G_2$ be finite groups acting primitively on sets $X_1$ and $X_2$ respectively.  Suppose $g_1\in G_1$ has two orbits on $X_1$, of sizes $f_1\ge e_1$ and $g_2\in G_2$ has two orbits on $X_2$ of sizes $f_2\ge e_2$.
If $G_1\cong G_2$, $1/e_1+1/f_1= 1/e_2+1/f_2\le 1$, and $G_1$ acts $2$-transitively on $X_1$, then $G_2$ acts $2$-transitively on $X_2$.
\end{prop}

\begin{proof}
As $1/e_1+1/f_1= 1/e_2+1/f_2\le 1$, we have $e_1,e_2\ge 2$.
Again we use Tables 1--3. Assuming for the sake of contradiction that $G_2$ does not act $2$-transitively on $X_2$, $e_2+f_2$ must be $10$ 
or a perfect square, so $(G_2,X_2)$ must be an affine action of type (c) or (e), a product action, or a sporadic action of type (b).

If $G_2$ is affine of type (c), then it is solvable, so $G_1\cong G_2$ implies $G_1$ is also affine, either one of the first two subcases of (b) or the second subcase of (e).
None of these cases match case (c) for both group order and $1/e_1+1/f_1$.
Next, if $G_2$ is affine of type (e), then it still  has a non-trivial solvable normal subgroup, so $G_1$ still must be affine. The possibilities for $1/e_2+1/f_2$ 
are $1$, $4/7$, $1/2$, $1/3$, $1/4$. Examining the possible cases for $G_1$ with these values for $1/e_1+1/f_1$, we see that $|G_1| \neq |G_2|$ in every case.

In Table 2, $G_2$ always has two simple factors in its composition series, so it cannot be isomorphic to any $G_1$, which belongs to Table 1 or 3 and therefore has at most one
simple factor.

\end{proof}

\section{The main theorems}

In this section, we prove the main results of the paper. 

We begin by analyzing how the coefficients of $\gamma_1(u)$ and $\gamma_2(u)$ depend on the coefficients $\kappa_k$ of $\chi(t)$.
Let $W$ denote the ring of polynomials in variables $z_1,z_2,z_3,\ldots$. 
If $a\in \Q$ and $b\in \C^\times$, a \emph{series of type $(a,b)$} in $t$ 
means a Puiseux series $F(t)$ with coefficients in $W$ of the form
$$F(t)= b t^a  + \sum_{k=1}^\infty w_k t^{a+k}$$
where each $w_k$ is a non-zero constant multiple of $z_k$ plus a polynomial  in $z_1,\ldots,z_{k-1}$. 
The \emph{specialization of $F$ to $(c_1,c_2,\ldots)$} is 
the Puiseux series $bt^a + \sum_{k=1}^\infty b_k t^{a+k}$ obtained by substituting $c_i$ for $w_i$ in $F$ for all $i\ge 1$.

\begin{lem}
\label{solve}
The coefficients $b_1,\ldots,b_k$ of the specialization of $F$ to $(c_1,c_2,\ldots)$ uniquely determine
the values $c_1,c_2,\ldots,c_k$.
\end{lem}
\begin{proof}
They are  determined by solving the system of equations 
$$w_1(z_1) = b_1,\ w_2(z_1,z_2) = b_2,\ldots,\ w_k(z_1,\ldots,z_k) = b_k$$
iteratively for $z_1,z_2,\ldots,z_k$.
\end{proof}

\begin{prop}
\label{standard}
We have the following formal statements:
\begin{enumerate}
\item If $F(t)$ is of type $(a,b)$ with $b>0$ and $m$ is a non-zero integer, then $F(t)$ has an $m^{\text{th}}$ root which is a series of type $(a/m,b^{1/m})$,
where $b^{1/m}$ denotes the positive $m^{\text{th}}$ root of $b$.
\item If $F(t)$ is a series of type $(1,1)$ (which, in particular, means it is a power series), then $F^{-1}(t)$ is a  series of type $(1,1)$.
\item If $F(t)$ is a series of type $(a,b)$ and $G(t)$ is a series of type $(c,d)$, with $c$ a positive integer, then $G(F(t))$ is a series of type $(ac, d b^c)$.
\end{enumerate}
\end{prop}

\begin{proof}
We use throughout the proof the fact that if $\sum_{i=0}^\infty v_i t^i$ and $\sum_{j=0}^\infty w_j t^j$ are power series in $t$ where $v_i$ and $w_i$ are
both polynomials in $z_1,\ldots,z_i$, then their product is again such a polynomial.

For part (1), we write $F(t) = b t^a (1 + \sum_{k=1}^\infty (w_k/b) t^k)$, and use Newton's binomial theorem to write
$$F(t)^{1/m} = b^{1/m} t^{a/m}\Big(1+\sum_{l=1}^\infty \binom{1/m}l\Bigl(\sum_{k=1}^\infty (w_k/b) t^k\Bigr)^l\Bigr).$$
For $l\ge 2$, expanding $\Bigl(\sum_{k=1}^\infty (w_k/b) t^k\Bigr)^l$, the $t^r$ coefficient is a polynomial in the variables $z_1,\ldots,z_{r-1}$ for $r\ge 1$.
The $l=1$ contribution gives a $t^r$ coefficient which is a non-zero multiple of $z_r$ plus a polynomial in $z_1,\ldots,z_{r-1}$. This gives (1).

For part (2), we use the Lagrange inversion formula \cite[Theorem 5.4.2]{Stanley}.
The $t^r$ coefficient of $F^{-1}(t)$ is $1/r$ times the $t^{r-1}$ coefficient of 
$$(F(t)/t)^{-r} = (1 + \sum_{k=1}^\infty w_k t^k)^{-r} =  1 + \sum_{l=1}^\infty \binom{-r}l \Bigl(\sum_{k=1}^\infty w_k t^k\Bigr)^l,$$
which is$\binom{-r}{r-1}z_{r-1}$ plus a polynomial in $z_1,\ldots,z_{r-2}$ exactly as before.

For part (3), if for all $i\ge 1$, $v_i$ denotes the $t^i$ coefficient of $F$ and $w_i$ the $t^i$ coefficient of $G$, then 
$$G(F(t)) =  d F(t)^c+\sum_{l=1}^\infty w_i F(t)^{c+i} = d(bt^a)^c(1+ \sum_{m=1}^\infty (v_m/b) t^m)^c + \sum_{l=1}^\infty w_i F(t)^{c+l}.$$
The contribution to $t^r$ from the first summand includes the $t^r$-term in $c\sum_{m=1}^\infty (v_m/b)t^m$ together with terms from higher powers of $\sum_{m=1}^\infty (v_m/b) t^m$,
for which the coefficient is a polynomial in $z_1,\ldots,z_{r-1}$. The contribution from the $l$th term of the second summand to the $t^r$ term can be expressed as a polynomial in $z_1,\ldots,z_{r-al}$.
In total, therefore, the $t^r$ term is a non-zero constant multiple of $z_r$ plus a polynomial in $z_1,\ldots,z_{r-1}$.
\end{proof}

Note that in each case in the above result, if a specialization of input series $F(t)$ (as well as $G(t)$ for part (3)) converges in a non-trivial disk, then the same specialization of the
output series converges in a disk.

\begin{lem}
\label{alpha2}
The function $\alpha_2(u)$ is given in a neighborhood of $+\infty$ on the real line by a power series in $u^{-1/e}$
which is obtained by taking a power series of type $(1,1)$ depending only on $e$, evaluating at $u^{-1/e}$,
multiplying by $c:=\kappa_{-e}^{1/e}$, and specializing to 
\begin{equation}
\label{special z}
(z_1,z_2,\ldots) = (\kappa_{1-e} c^{1-e},\kappa_{2-e} c^{2-e},\ldots,\kappa_f c^f,0,0,\ldots)
\end{equation}
\end{lem}

\begin{proof}
By Proposition \ref{standard}~(1), there exists a type $(1,1)$ series $H(t)$ such that
$$H(t)^{-e} = t^{-e}\Bigl(1+\sum_{k=1}^\infty z_k t^k\Bigr).$$
The specialization of $F(t) := H(c^{-1} t)^{-e}$ to \eqref{special z} is
$$\kappa_{-e}t^{-e}(1 + \sum_{k=1}^{e+f} \kappa_{k-e} c^{k-e} (c^{-1}t)^k) = \kappa_{-e}t^{-e} + \sum_{k=1}^{e+f} \kappa_{k-e} t^{k-e} = \chi(t).$$

By Proposition \ref{standard}~(2), the formal inverse $G(t)$ of $H(t)$ is given by a series in $t$ of type $(1,1)$, so 
$$F(cG(u^{-1/e})) = H(G(u^{-1/e}))^{-e} = (u^{-1/e})^{-e} = u.$$
The function $G$ now has the desired properties.

\end{proof}

\begin{lem}
\label{tchiprime}
There exists a series in $t$ of type $(e,-1/e\kappa_{-e})$ whose specialization to \eqref{special z} evaluated on $c^{-1}t$
gives $1/t \chi'(t)$ in a neighborhood of $0$.
\end{lem}

\begin{proof}
The specialization of
$$F(t) := -e\kappa_{-e}t^{-e}\Bigl(1 + \sum_{k=1}^\infty \frac{e-k}e z_kt^k\Bigr)$$
to \eqref{special z} evaluated at $c^{-1} t$ gives $t \chi'(t)$. By Proposition \ref{standard}~(1), the reciprocal of $F(t)$
is of type $(e,-1/e\kappa_e)$, and its specialization, evaluated at $c^{-1}t$, gives the power series expansion
of $\frac 1{t\chi'(t)}$.
\end{proof}

\begin{prop}
There exists a series in $t$ of type $(e,-1/e)$ whose specialization to \eqref{special z} when evaluated at 
$u^{1/e}$ gives $\gamma_2(u)$ in a neighborhood of $+\infty$ on the real line.
\end{prop}

\begin{proof}
By Lemmas \ref{alpha2} and \ref{tchiprime}, the series for $\gamma_2(u) = \frac 1{\alpha_2(u)\chi'(\alpha_2(u))}$ is 
$$F(c^{-1}cG(u^{-1/e})) = F(G(u^{-1/e})).$$
By Proposition~\ref{standard}~(3), this is given by specializing a series of type 
$$(e,-c^e/e\kappa_{-e}) = (e,-1/e)$$
in $u^{-1/e}$ to \eqref{special z}.
\end{proof}

We therefore have an expansion 
\begin{equation}
\label{gamma2}
\gamma_2(u) \in -\frac{u^{-1}}e + u^{-1-1/e}\C[[u^{-1/e}]].
\end{equation}
We can calculate $\gamma_1(u)$ in the same way, regarding $\chi(t)$ as a Laurent series in $1/t$ with leading term $\kappa_f (1/t)^{-f}$, so
$\alpha_1(u)$ has a Laurent series expansion in $u^{-1/f}$ with leading term $\kappa_f^{-1/f}u^{1/f}$,
and
\begin{equation}
\label{gamma1}
\gamma_1(u) \in  \frac {u^{-1}}f  + u^{-1-1/f} \C[[u^{-1/f}]].
\end{equation}

Recall that we have defined $\chi^+(t)=\sum_{k\ge 1}\kappa_k t^k$ and $\chi^-(t) := \sum_{k\le -1} \kappa_k t^k$
and defined by $[\chi^+]$ and $[\chi^-]$ the equivalence class of these functions under linear rescaling of $t$.

\begin{prop}
\label{scale}
If $\chi_1$ and $\chi_2$ are Laurent polynomials with non-negative coefficients satisfying $\chi_i(1) = 1$, $\chi'_i(1)=0$,
$[\chi_1^+] = [\chi_2^+]$, $[\chi_1^-] = [\chi_2^-]$, and the constant terms of $\chi_1$ and $\chi_2$ are equal, then $\chi_1=\chi_2$.
\end{prop}

\begin{proof}
As the constant coefficients of $\chi_1$ and $\chi_2$ are the same, $\chi_1^+(1) + \chi_1^-(1) = \chi_2^+(1)+\chi_2^-(1)$.
We have $\chi_1^+(t) = \chi_2^+(at)$ and $\chi_1^-(t) = \chi_2^-(bt)$ for all $t>0$, where without loss of generality we may assume $a\ge 1$.
Thus,
$$\chi_1^-(1) = \chi_2^-(1) + \chi_2^+(1) - \chi_1^+(1) = \chi_2^-(1) + \chi_2^+(1) - \chi_2^+(a) \le \chi_2^-(1) = \chi_1^-(1/b).$$
As $\chi_1^-$ is decreasing, this implies $b\ge 1$.

Since $\chi_i^{\pm}$ are convex, 
\begin{equation}
\label{prime1}
(\chi_1^+)'(1) = a(\chi_2^+)'(a) \ge (\chi_2^+)'(a) \ge (\chi_2^+)'(1)
\end{equation}
and
\begin{equation}
\label{prime2}
(\chi_1^-)'(1) \ge (\chi_1^-)'(1/b) = b(\chi_2^-)'(1) \ge (\chi_2^-)'(1).
\end{equation}
However, $(\chi_i^-)'(1) + (\chi_i^+)'(1) = \chi'_i(1) = 0$ for $i=1,2$, so equality must hold in both \eqref{prime1} and \eqref{prime2},
implying $a=b=1$.
\end{proof}

\begin{prop}
\label{e=1}
If two unbiased random walks have the same spectrum and have support bounded below by $-1$, then the two walks are equivalent.
\end{prop}

\begin{proof}
The constant coefficient of $(\kappa_{-1}t^{-1} + \kappa_0+\kappa_1 t + \cdots)^k$ can be expressed as $k (\kappa_{-1}^{k-1}\kappa_{k-1})$
plus a polynomial expression in terms of 
$$\kappa_0, \kappa_{-1} \kappa_1, \kappa_{-1}^2\kappa_2,\ldots,\kappa_{-1}^{k-2} \kappa_{k-2}.$$
Therefore, we can solve iteratively for $\kappa_0, \kappa_{-1}\kappa_1,\ldots,\kappa_{-1}^f \kappa_f$ given $I_1, I_2,\ldots, I_{f+1}$.  
This implies $I_\bullet$ determines the equivalence class $[\chi^+]$ of $\chi^+(t) := \sum_{k=0}^f \kappa_k t^k$.
The result now follows immediately from Proposition~\ref{scale}.
\end{proof}

We can now prove Theorem~\ref{Main}.

\begin{proof}
By Propositions~\ref{L asymptotics} and \ref{tilde asymptotics}, the spectrum $I_\bullet$ of $\chi$ determines
the Puiseux series for
$\gamma_1 - \gamma_2$, which by \eqref{gamma2} and \eqref{gamma1} can be written
$$\Bigl(\frac 1e+\frac 1f\Bigr) u^{-1} + \sum_{q\in (-\infty,-1)\cap \Q} c_q u^q.$$
The $u^{-1}$ coefficient is greater than $1$ if and only if $e=1$, in which case Proposition~\ref{e=1} implies that
$\chi$ is determined by $I_\bullet$.  It remains to consider the case $e>1$.

Let $K:=\C(t)$ and $K_\infty := \C((t^{-1}))$.
Fix once and for all an embedding of $\overline{K}\subset \overline{K_\infty}$.
The Puiseux series determined by $I_\bullet$ determines an element $\gamma_1-\gamma_2$ of $\overline K$.

Let $M$ denote the splitting field of the minimal polynomial of this element over $K$.
Let $\delta_1$ and $\delta_2$ denote any two elements of $M$ conjugate over $M$ such that $M$ is the splitting field of $K(\delta_1)$ over $K$, $K(\delta_1)$ is a minimal non-trivial extension of $K$,
and $\delta_1-\delta_2 = \gamma_1-\gamma_2$.
By Proposition~\ref{Forced 2-transitivity}, the action of $\Gal(M/K)$ on the set of conjugates of $\delta_1$ is $2$-transitive.
By Propositions \ref{two 2-trans} and \ref{two-orbit}, either $\gamma_1 \in \delta_1 + K$ or $\gamma_1\in -\delta_2+K$. 
We can regard $\delta_1$ and $\delta_2$, in some order, as Laurent series in $u^{-1/e}$ and 
$u^{-1/f}$ which are well-defined modulo $K_\infty$. If $f>e$, choose the numbering so that $\delta_1$ is a Laurent series in $u^{-1/f}$
and $\delta_2$ is a Laurent series in $u^{-1/e}$

Replacing $\chi(t)$ by $\chi(t^{-1})$ if necessary, we may assume $\delta_i\cong \gamma_i\pmod{\C((u)}$, which means that $\delta_1$ determines the coefficients of $u^{-1-1/e},u^{-1-2/e},\ldots,u^{-2+1/e}$
of $\gamma_1$ and $\delta_2$ determines the coefficients $u^{-1-1/f},u^{-1-2/f},\ldots,u^{-2+1/f}$ of $\gamma_2$.
By Lemma~\ref{solve} and \eqref{special z}, this means that $\delta_1$ and $\delta_2$ determine, respectively, the sequences
$$\kappa_{f-1}\kappa_f^{-1+1/f},\kappa_{f-2}\kappa_f^{-1+2/f},\ldots,\kappa_1\kappa_f^{-1/f}$$
and
$$\kappa_{1-e}\kappa_{-e}^{-1+1/e},\kappa_{2-e}\kappa_{-e}^{-1+2/e},\ldots,\kappa_{-1}\kappa_{-e}^{-1/e}.$$
Since $\kappa_0 = I_1$, $\delta_1$ and $\delta_2$ determine $[\chi^+]$ and $[\chi^-]$ respectively. By Proposition~\ref{scale},
these determine $\chi$.

\end{proof}

\begin{cor}
If $\rho\colon U(1)\to \SL(V)$ is a complex representation whose formal character is primitive and neither of degree $10$ nor of square degree, then
$\rho(U(1))$ is determined up to conjugation in $\SL(V)$ by the function $n\mapsto \dim (V^{\otimes n})^{U(1)}$.
\end{cor}

\begin{proof}
Defining $\chi(t)$ to be $(\dim V)^{-1}$ times the formal character of $\rho$, regarded as a Laurent polynomial in $t$, we see that $\chi(1) = 1$, and $\det(\rho)=1$ implies $\chi'(1)=0$.
Moreover, 
$$I_n = \frac{ \dim (V^{\otimes n})^{U(1)}}{(\dim V)^{n}}$$
is determined, so by Theorem~\ref{Main}, $\chi(t)$ is determined. Moreover,
$$\dim V = \limsup_{n\to \infty} \bigl(\dim (V^{\otimes n})^{U(1)}\bigr)^{1/n},$$
so this determines the formal character of $\rho$. The formal character determines $\rho(U(1))$ up to conjugation by $\GL(V)$, but since every element of $\GL(V)$ is a scalar multiple of an element of $\SL(V)$,
it determines $\rho(U(1))\subset \SL(V)$ up to conjugation.
\end{proof}

The following lemma is needed for the proof of Theorem~\ref{clean}.

\begin{lem}
\label{bias}
An unbiased random walk cannot have the same spectrum as a biased walk.
\end{lem}

\begin{proof}
If $\chi(t)$ is the shape of a biased random walk, then $\chi'(1) \neq 0$, so there exists $\lambda > 0$ with 
$\chi(\lambda) = r < 1$.  It follows that for all $\epsilon > 0$,
$$\int_{-\infty}^\infty e^{-s(\chi(e^x)-1)}dx > e^{s(1-r - \epsilon)}$$
if $s$ is sufficiently large.  By Proposition~\ref{tilde asymptotics}, it follows that the spectrum of a biased random walk cannot be the same as the spectrum of an unbiased walk.
\end{proof}

Next, we prove Theorem~\ref{clean}.

\begin{proof}
Suppose $\omega(t)$ is the shape of a random walk with the same spectrum as $\chi(t)$.  
By Lemma~\ref{bias}, $\omega(t)$ is also unbiased.  Let $d$ be the greatest common divisor of $\supp(\omega)$.  
Replacing $\omega(t)$ with $\omega(t^{1/d})$,
we may assume without loss of generality that $\supp(\omega)$ generates $\Z$.  The results of \S3 now apply to $\omega$ as well as $\chi$, so 
$\gamma_1-\gamma_2$ gives the same Puiseux series in $u$ as the corresponding expression for $\omega$.  
Therefore, the groups $G$ of the corresponding Galois extensions of $\C(u)$
are isomorphic.  
However, as the maximum and minimum elements of $\supp(\chi)$ are relatively prime to one another, $\chi$ must be primitive.
By \cite[Theorem 6.2]{Muller0}, therefore,
either $e=1$, $G=A_n$, or $G=S_n$, where $n=\deg(\chi) = \deg(\omega)$,
and these cases are covered by the proof of Theorem~\ref{Main}.

\end{proof}

Recall that $X_{e,f}$ is the variety of dimension $e+f-1$ defined by the conditions \eqref{affine} on $\kappa_{-e},\ldots,\kappa_f$.
For Theorem~\ref{generic}, we need the following proposition.

\begin{prop}
\label{trans}
For all $e,f\ge 1$, there is a dense Zariski open subset of $X_{e,f}$ all of whose points correspond to Laurent polynomials
$\chi$ such that the Galois group $G$ of the splitting field of $L/K$ contains a transposition.
 \end{prop}

\begin{proof}
We claim that the set of $\chi\in X_{e,f}$ such that the map $t\mapsto \chi(t)$ has a critical value in $\C$ with ramification degree $>1$ is contained in a proper Zariski-closed subset of $X_{e,f}$. Indeed, 
if $\chi$ has such a critical value $u_0$, then
$\chi(t)-u_0$ must have at least two double zeroes or a zero of order $\ge 3$.  If $t_1$ and $t_2$ are double zeroes, then
the condition $\chi(1) = 1$ implies
$$\chi(t) = \frac{(t-t_1)^2(t-t_2)^2 P(t)}{t^e} + 1 - (1-t_1)^2(1-t_2)^2P(1)$$
for some polynomial $P(t)$ of degree $e+f-4$.  The condition $\chi'(1) = 0$ gives the constraint
\begin{align*}
2(1-t_1)^2(1-t_2)P(1) &+ 2(1-t_1)(1-t_2)^2P(1) \\
& -e (1-t_1)^2(1-t_2)^2 P(1)+ (1-t_1)^2 (1-t_2)^2 P'(1) = 0,
\end{align*}
which for any given pair $(t_1,t_2)$ is a single linear constraint of the coefficients of $P$.  Except when $t_1$ or $t_2$ equals $1$, this constraint is non-trivial.
Therefore, the number of degrees of freedom needed to determine such a $\chi$ is $e+f-2 < \dim X_{e,f}$.

Likewise, if $\chi(t)-u_0$ has a triple zero at $t_1$, $\chi$ must be of the form
$$\frac{(t-t_1)^3 P(t)}{t^e} + 1 - (1-t_1)^3P(1),$$
where $P$ is a polynomial of degree $e+f-3$ satisfying
$$3(1-t_1)^2 P(1) - e(1-t_1)^3 P(1) + (1-t_1)^3 P'(1) = 0,$$
which is a non-zero constraint on $P$ except when $t_1=1$.  Therefore, again, the number of degrees of freedom in such a $\chi$ is $e+f-2$.

For any $\chi$ which does not have such a $u_0$, any critical value of $\chi$ in $\C$ is a point over which the map $t\mapsto \chi(t)$ has a single double point and no other critical point.  The local monodromy around that point is therefore a transposition. 
\end{proof}

\begin{lem}
\label{generic primitive}
For $f\ge e\ge 2$,
there exists a Zariski-dense open subset of $X_{e,f}$ such that every $\chi(t)$ belonging to this set is primitive.
\end{lem}

\begin{proof}
If $\chi$ is not primitive, it is a composition $f(g(t))$ of Laurent polynomials, where we may assume $g(t)$ has leading coefficient $1$ and constant term $0$.
Moreover, $f(g(1)) = 1$, and $f'(g(1)) = 0$ or $g'(1)=0$.  If $f$ and $g$ have degree $a>1$ and $b>1$ respectively, we have $ab = e+f$.
Therefore, we have $(a+1)+(b+1)-4 = a+ b - 2$ degrees of freedom in choosing $(f,g)$, while $X_{e,f}$ has dimension 
$$e+f-1 = ab-1 = (a-1)(b-1)+a+b-2 > a+b-2.$$
\end{proof}

We easily deduce Theorem~\ref{generic} from these two results.  
\begin{proof}
Let $U$ denote a dense open set of $X_{e,f}$ for which every $\chi\in U$ is primitive and the Galois group $G$ of the splitting field of $\C(t)/\C(\chi(t))$ contains a transposition.
A primitive permutation group which contains a transposition must be the full symmetric group.
As we have already seen, in this case, the spectrum $I_\bullet$ determines $\chi(t)$ up to equivalence.

\end{proof}

\end{document}